\documentclass{amsart}
\title{On the values taken by slice torus invariants}
\author{Peter Feller}
\address{ETH Zurich, R\"amistrasse 101, 8092 Zurich, Switzerland}
\email{\myemail{peter.feller@math.ch}}
\urladdr{\url{https://people.math.ethz.ch/~pfeller/}}
\author{Lukas Lewark}
\address{Faculty of Mathematics, University of Regensburg, 93053 Regensburg, Germany}
\email{\myemail{lukas@lewark.de}}
\urladdr{\url{http://www.lewark.de/lukas/}}
\author{Andrew Lobb}
\address{Mathematical Sciences, Durham University, UK}
\email{\myemail{andrew.lobb@durham.ac.uk}}
\urladdr{\url{http://www.maths.dur.ac.uk/users/andrew.lobb/}}
\usepackage{thmtools,enumerate,graphicx,xcolor,enumitem,afterpage,float,ucs,amsmath,amssymb,amscd,url,mathtools}
\usepackage{pinlabel}
\usepackage{enumerate}
\usepackage[latin1]{inputenc}
\usepackage[all]{xy}
\usepackage{amsthm}
\usepackage[hypertexnames=false]{hyperref}
\usepackage[nameinlink]{cleveref}
\let\cref\Cref
\crefname{subsection}{subsection}{subsections}
\Crefname{subsection}{Subsection}{Subsections}
\Crefname{enumi}{}{}
\crefname{equation}{}{}
\definecolor{darkblue}{RGB}{0,0,96}
\definecolor{gray}{RGB}{127,127,127}
\definecolor{darkred}{RGB}{160,0,0}
\definecolor{lightyellow}{RGB}{255,255,128}
\hypersetup{colorlinks={true},linkcolor={black},citecolor={black},filecolor={black},urlcolor={black},pdfauthor={\authors},pdftitle={\shorttitle}}
\newcommand{\myemail}[1]{\href{mailto:#1}{#1}}
\newcommand{\qua}{\hskip 0.4em \ignorespaces}
\def\arxiv#1{\relax\ifhmode\unskip\qua\fi
\href{http://arxiv.org/abs/#1}%
{\tt arXiv:\penalty -100\unskip#1}}

\def\MR#1{\relax\ifhmode\unskip\qua\fi
\href{http://www.ams.org/mathscinet-getitem?mr=#1}{\tt MR#1}}
\def\xox#1{\csname xx#1\endcsname}

\declaretheorem{lemma}
\newtheorem{theorem}[lemma]{Theorem}

\newtheorem{proposition}[lemma]{Proposition}
\newtheorem{conjecture}[lemma]{Conjecture}

\newtheorem*{prize*}{Prize}
\newtheorem*{theorem*}{Theorem}
\newtheorem{question}[lemma]{Question}
\newtheorem*{question*}{Question}
\theoremstyle{definition}
\newtheorem{definition}[lemma]{Definition}

\newtheorem{remark}[lemma]{Remark}

\newtheorem{example}[lemma]{Example}
\DeclareMathOperator{\cl}{cl}

\DeclareMathAlphabet{\mathpzc}{OT1}{pzc}{m}{it}

\newcommand{\cC}{\mathcal{C}}

\newcommand{\Z}{\mathbb{Z}}

\newcommand{\Q}{\mathbb{Q}}
\newcommand{\R}{\mathbb{R}}

\newcommand{\Hom}{\text{Hom}}

\begin{document}
\thispagestyle{empty}
\begin{abstract}
	We study the space of slice-torus invariants.  In particular we characterize the set of values that slice-torus invariants may take on a given knot in terms of the stable smooth slice genus.
        Our study reveals that the resolution of the local Thom conjecture implies the existence of slice torus invariants without having to appeal to any explicit construction from a knot homology theory.
\end{abstract}
\maketitle
\section{Introduction}
A fruitful approach to understanding a group is to construct homomorphisms on it; the group of interest in this paper is the smooth concordance group $\cC$ of knots.
A classical example of such homomorphisms is given by the Levine-Tristram signatures $\cC \rightarrow \R$, which were used in 
Litherland's proof
that positive non-trivial torus knots $T_{p,q}$ are linearly independent in $\cC$ \cite{MR547456}.
Signatures also provide lower bounds for the smooth slice genus $g_4(K) \in \Z_{\geq 0}$ of a knot $K$
but, in the case of torus knots, these bounds are not sufficient to determine $g_4(T_{p,q})$.
In fact, that we have $g_4(T_{p,q}) = (|p|-1)(|q|-1)/2$, which is known as the local Thom conjecture, was first shown
by Kronheimer and Mrowka, as a consequence of their resolution of the Thom conjecture~\cite{thom} using gauge theory.
This article is concerned with a class of homomorphisms $\cC\to\R$ that is much younger than signatures,
namely slice-torus invariants, whose definition goes back to Livingston~\cite{livingston} (see also \cite{lew2}).
\begin{definition}
	\label{defn:slice_torus}
	A \emph{slice-torus invariant} is a homomorphism $\phi \colon \cC \rightarrow \R$ satisfying two conditions:\\[1ex]
	\begin{tabular}{@{\hspace{3em}}l@{\ }lr@{\ }ll}
		    & \textsc{Slice:} & $\phi(K)$ & $\leq g_4(K)$ & for all knots $K$ and \\
		    & \textsc{Torus:} & $\phi(T_{p,q})$ & $= g_4(T_{p,q})$ & for all positive coprime integers $p,q$.
	\end{tabular}
\end{definition}
Note that it is quite non-trivial that such invariants do exist.
Using suitable normalizations, the first slice-torus invariant to be constructed was the $\tau$ invariant coming from knot Floer homology \cite{osz10,ras},
followed by the Rasmussen invariant $s$ coming from Khovanov homology~\cite{ras3},
and the $s_n$ invariants coming from $\mathfrak{sl}_n$ Khovanov-Rozansky homologies \cite{wu3,lobb1,lobb2}.
We study the set $V \subset \Hom(\cC, \R)$ of all slice torus invariants.
Note that $V$ is non-empty and convex.
It follows that for each $K\in \cC$, the set $V(K) \coloneqq \{\phi(K)\mid\phi\in V\}\subset\R$ is a nonempty interval.
The main result of this note provides a description of these intervals,
in terms of the \emph{stable smooth slice genus} (compare~\cite{livingston_stable})
$\displaystyle\widehat{g_4}(K) \coloneqq \lim_{n\to\infty} g_4(K^{\# n})/n$.

\begin{proposition}\label{lemma:stivalues}
For every knot $K$, the sequence $t_1(K), t_2(K), \ldots$ defined as
\[
t_p(K) \coloneqq
\widehat{g_4}(T_{p,p+1} \# K)
- \widehat{g_4}(T_{p,p+1})
\]
is decreasing and convergent.
Its limit $\ell(K)$ satisfies $-\ell(-K)\leq\ell(K)$.
\end{proposition}
\begin{theorem}\label{thm:stivalues}
For every knot $K$,
the set $V(K) = \{\phi(K)\mid\phi\in V\}\subset\R$
of values taken by all slice-torus invariants on $K$ equals $[-\ell(-K), \ell(K)]$. 
\end{theorem}

\begin{remark}\label{rmk:onlyp}
Our proof of \cref{thm:stivalues} uses the fact that $g_4(T_{p,p+1})=p(p-1)/2$ for all integers ${p\geq 1}$~\cite{thom}, but we do not use the \emph{a priori} existence of any slice torus invariant.
Thus, it follows from our proof that the local Thom conjecture implies the existence of slice torus invariants without the need of any explicit construction of a slice torus invariant. However, we note that from our proof it is not clear that there exist integer valued slice torus invariants such as $s$ or $\tau$ (suitably normalized), or even that $[-\ell(K), \ell(-K)]$ contains an integer for all knots $K$.
\end{remark}

\begin{example}
Let us explicitly calculate $V(K)$ for $K$ the $(2,-3,5)$ pretzel knot, which is $10_{125}$ in the knot table.
As above, let $s_n$ be the concordance invariant coming from $\mathfrak{sl}_n$ Khovanov-Rozansky homology.
Then $\tilde{s}_n \coloneqq s_n/2(n-1)$ is a slice-torus invariant, and $s_2$ is equal to the original Rasmussen invariant $s_2 = s$.
One may calculate that $\tilde s_2(K) = 1$
and furthermore, for all $n \geq 3$,
that $\tilde s_n(K) \in \{0, 1/(n-1)\}$; see~\cite{lew2}.
Since $\lim_{n\to \infty} \tilde s_n(K) = 0$,
it follows from \cref{thm:stivalues} that $[0,1] \subset V(K)$.

To show the converse inclusion $V(K) \subset [0,1]$, let us use the sharpened slice-Bennequin inequalities
\cite{lobb5,lew2}. We will only need the braid version of the inequalities as stated in \cref{eq:bennequin} below.
Denote by $\sigma_1, \ldots, \sigma_{k-1}$ the standard generators of the braid group $B_k$ on $k$ strands.
For $\beta$ a word in these generators, let
\[
O_{\pm}(\beta) = \# \{i \in \{1, \ldots, k -1 \}\mid \sigma_i^{\pm 1} \text{ does not appear in } \beta\}.
\]
Then for all slice-torus invariants $\phi\in V$, we have that the value taken on the closure~$\cl(\beta)$ satisfies
\begin{equation}\label{eq:bennequin}\tag{$*$}
2\phi(\cl(\beta)) \quad\in\quad [1 + w(\beta) - k + 2O_+(\beta),\ -1 + w(\beta) + k - 2O_-(\beta)],
\end{equation}
where $w(\beta)$ denotes the writhe of $\beta$.
The knot $K = 10_{125}$ is the closure of $\beta = \sigma_1^5 \sigma_2^{-1} \sigma_1^{-3} \sigma_2^{-1}\in B_3$.
Since $w(\beta) = 0$, $O_+(\beta) = 1$ and $O_-(\beta) = 0$, \cref{eq:bennequin} implies that
$\phi(K) \in [0, 1]$ for all $\phi \in V$, and thus $V(K) \subset [0,1]$.
All in all, we have shown that $V(K) = [0,1]$.

This computation of $V(K)$ yields further examples. Namely,
for all $a,b\in \Z$ with $a \geq 0$ it immediately follows that
\[
V(K^{\# a} \# T_{2,3}^{\# b}) = [b,a + b].
\]
Hence we have the following result.
\end{example}
\begin{proposition}\label{prop:integerrealize}
        Every nonempty compact interval with integral endpoints is realized as $V(J)$ for some knot $J$.\qed
\end{proposition}
Beyond \cref{prop:integerrealize}, we do not know if any further intervals can be realized.
The following geography question thus remains open.

\begin{question}\label{q:intervals}
Which nonempty compact intervals arise as $V(J)$ for some knot $J$?
\end{question}
\subsection*{Acknowledgments}The first author gratefully acknowledges support by the SNSF Grant 181199.
The second author gratefully acknowledges support by the DFG, project no.~412851057.
\section{Squeezed knots}
Slice torus invariants all agree on the following class of knots.
\begin{definition}[\cite{sqz}]
A knot $K$ is called \emph{squeezed} if and only if
there exists a smooth oriented connected cobordism $C_+$ between a positive torus knot $T^+$ and $K$,
and a smooth oriented connected cobordism $C_-$ between $K$ and a negative torus knot $T^-$
such that $C_+ \cup C_-$ is a smooth oriented connected cobordism between $T^+$ and $T^-$ that is genus-minimizing.
\end{definition}
The reader may wish to try to prove the following proposition directly from the definitions.  It says, roughly speaking, that squeezed knots are boring from the point of view of slice-torus invariants.
\begin{proposition}[\cite{sqz}]
	\label{lem:slice_torus_boring_on_sqzd}
	If $\phi_1$ and $\phi_2$ are slice-torus invariants and $K$ is squeezed then we have that $\phi_1(K) = \phi_2(K)$.  
        \qed
\end{proposition}
By \cref{thm:stivalues}, \cref{lem:slice_torus_boring_on_sqzd} also follows from the following.
\begin{proposition}\label{lem:sqzell}
If a knot $K$ is squeezed, then $-\ell(-K) = \ell(K)$.
\end{proposition}
\begin{proof}
Let $C_{\pm}$ and $T^{\pm}$ be chosen as in the definition of squeezedness applied to $K$.
We may assume that for some large $p>0$, the $T^\pm$ satisfy $T^+ = T := T(p,p+1)$ and $T^- = -T$.  This is because for any positive torus knot $L$ there exists a $p > 0$ such there is a genus-minimizing slice surface for $T(p,p+1)$ that factors through~$L$ (see \cref{lem:add}~(i)).
Then, we have
\begin{align*}
\ell(-K) + \ell(K) & \leq t_p(-K) + t_p(K) \\
\intertext{by the monotonicity of $t_p(K)$ shown in \cref{lemma:stivalues}. By definition of $t_p$ this equals}
 & = \widehat{g_4}(T\# \,{-K}) + \widehat{g_4}(T\# K) - 2\widehat{g_4}(T). \\
\intertext{It is well known (and we provide a proof in \cref{lem:add}~(iv)) that slice genus and stable slice genus of torus knots agree; therefore we find the equality}
 & = \widehat{g_4}(T\# \,{-K}) + \widehat{g_4}(T\# K) - g_4(T\# T). \\
\intertext{Since $\widehat{g_4} \leq g_4$ for all knots,}
 & \leq g_4(T\# \,{-K}) + g_4(T\# K) - g_4(T\# T). \\
\intertext{For all $J, J'$, $g_4(J\#\,{-J'})$ equals the cobordism distance between $J$ and~$J'$, and so}
 & \leq g(C_+) + g(C_-) - g_4(T\# T),
\end{align*}
which equals $0$ because of the assumption that $C_+\cup C_-$ is a genus-minimizing cobordism between $T$ and $-T$.
\end{proof}
The proof of \cref{lem:sqzell} shows that if $K$ is squeezed,
then the sequence $t_p(K)$ is constant for sufficiently large $p$. We do not know whether this is the case for all knots:
\begin{question} 
	We ask the following.
	\begin{enumerate}[label=(\roman*)]
\item \label{i} Is $\ell(K)$ an integer for all knots $K$?

\item \label{ii}(Stronger) Does the sequence $\widehat{g_4}(T_{p,p+1} \# K)$ have only finitely many non-integer values for every fixed knot $K$?

\item \label{iii}(Strongest) Does $\widehat{g_4}(T_{p,p+1} \# K)=g_4(T_{p,p+1} \# K)$ hold for all but finitely many $p$ for every fixed knot?
\end{enumerate}
\end{question}
\begin{remark}
If~\Cref{i} can be answered positively, then~\Cref{q:intervals} is resolved: the intervals that occur as $V(J)$ for some knot are $[a,b]$ with $a\leq b$ integers.

If~\Cref{iii} 
can be answered positively, then $K$ satisfying $-\ell(-K) = \ell(K)$ implies that $K$ is squeezed. This is seen as follows.
If a knot $K$ satisfies $-\ell(-K) = \ell(K)$,
then we have for some $p>0$
\[ g_4(T_{p,p+1} \# K) - g_4(T_{p,p+1}) = \ell(K) = -\ell(-K) = g_4(T_{p,p+1}) - g_4 ( T_{p,p+1}  \# -\!K ) \]
and hence
\[ g_4(T_{p,p+1} \# K) + g_4 (T_{p,p+1} \# -\!K ) = 2 g_4 ( T_{p,p+1} ) {\rm .} \]
The left hand side of the equation is the genus of a cobordism from $T_{p,p+1}$ to $-T_{p,p+1}$ that factors through~$K$. On the other hand, the right hand side is the minimal genus of a cobordism from $T_{p,p+1}$ to $-T_{p,p+1}$.  Thus we see that~$K$ must be squeezed.
\end{remark}

In light of this, we conjecture the converse of \cref{lem:sqzell}.
\begin{conjecture} For all knots $K$, $K$ is squeezed if and only if $-\ell(K)=\ell(-K)$.\end{conjecture}

\section{Proof of the main theorem}
The stable 4-genus $\widehat{g_4}$ induces a seminorm on the vector space $\mathcal{C}\otimes \R$,
as discussed by Livingston \cite{livingston_stable} (Livingston states the result for the vector space $\mathcal{C} \otimes \Q$, but it easily extends to $\mathcal{C} \otimes \R$).
Moreover, every slice torus invariant $y$ gives rise to homomorphism $y\colon \mathcal{C} \otimes \mathbb{R} \to \mathbb{R}$, with $y\leq \widehat{g_4}$.
Here, our slightly abusive notation does not differentiate between $\widehat{g_4}$ and the induced seminorm,
$y$ and the induced homomorphism, nor between knots and the vectors they represent in $\mathcal{C}\otimes\R$.

In what follows, let $\mathcal{T}$ be the real subspace of $\mathcal{C}\otimes\mathbb{R}$ generated by torus knots,
and let $\mathcal{T}_+ \subset \mathcal{T}$ be the closed convex cone consisting of linear combinations of positive torus knots with non-negative coefficients.

Let us emphasize, as mentioned in \cref{rmk:onlyp}, that we only use the fact that $g_4(T_{p,p+1}) = p(p-1)/2$.
The realization that this fact is enough to determine $g_4$ for much larger classes of knots
is due to Rudolph~\cite{rudolph_QPasObstruction}.
For the sake of self-containedness, we include a short proof.
\begin{lemma}\label{lem:add}
	We have the following.
\begin{enumerate}[label=(\roman*)]
\item For all knots $K\in\mathcal{T}_+$, there is an integer $p \geq 1$ and a smooth cobordism~$C$
between $K$ and $T_{p,p+1}$ such that $g(C) = g_4(T_{p,p+1}) - g_4(K)$.
\item For all integers $p \geq 2$, there is a smooth cobordism~$C$ between $T_{p-1,p}$ and $T_{p,p+1}$ of genus $p-1$.
\item For all knots $K, K'\in \mathcal{T}_+$, we have $g_4(K \# K') = g_4(K) + g_4(K')$.
\item For all knots $K \in \mathcal{T}_+$, we have $\widehat{g_4}(K) = g_4(K)$.
\end{enumerate}
\end{lemma}
\begin{proof}
\noindent (i) Since $K$ is a connected sum of positive torus knots, it may in particular be written as closure
of a positive braid word $\beta\in B_k$ for some $k$. Assume that $\beta$ is the product of $l$ generators.
Replace each $\sigma_i$ in $\beta$ with $\sigma_1 \cdots \sigma_{k-1}$ to find a cobordism (consisting of $l(k-2)$ $1$-handles) from $K$ to the torus link $T(k,l)$.
Set $p = \max\{k, l-1\}$.
Compose this first cobordism with a cobordism from $T(k,l)$ to $T(k,p+1)$ given by $(p+1-l)(k-1)$ $1$-handles, and then with a further cobordism from $T(k,p+1)$ to $T(p,p+1)$ given by $(p-k)p$ $1$-handles.
In total, this yields a cobordism $C$ of genus $g(C) = p(p-1)/2 - (1 + l - k)/2$.
The triangle inequality implies $g_4(K) \geq g_4(T_{p,p+1}) - g(C) = (1 + l - k)/2$.
On the other hand, Seifert's algorithm applied to $\beta$ results in a Seifert surface of
genus $(1 + l - k)/2$ for $K$. Thus, $g_4(K) = (1 + l - k)/2$, and $g(C) = g_4(T_{p,p+1}) - g_4(K)$ as desired.\medskip

\noindent (ii) Note that $T_{p-1,p}$ is the closure of the braid $\beta = (\sigma_1\cdots \sigma_{p-1})^{p-1} \in B_{p}$.
The desired cobordism $C$ consists of $2(p-1)$ 1-handles and may be constructed by
appending $(\sigma_1\cdots \sigma_{p-1})^2$ to $\beta$, thus obtaining the braid $(\sigma_1\cdots \sigma_{p-1})^{p+1}$,
whose closure is $T_{p,p+1}$.
The existence of $C$ is also implicit in \cite[Proof~of~Theorem~2]{zbMATH06092636}, \cite[Example~20]{zbMATH06723030}, or follows from \cite[Theorem~2]{zbMATH06272609}.\medskip

\noindent (iii) As in (i), $K$ and $K'$  may be written as closures of positive braid words $\beta \in B_k, \beta' \in B_{k'}$ that are the product of $l$ and $l'$ generators, respectively. Then, $K \# K'$ is the closure of a positive braid word $\beta''\in B_{k + k' - 1}$ that is the product of $l + l'$ generators. As shown in (i), this implies that
\begin{align*}
g_4(K \# K') & = \frac{1 + l + l' - (k + k' - 1) -1 }2 \\
 & = \frac{1 + l - k}2 + \frac{1 + l' - k'}2 = g_4(K) + g_4(K').
\end{align*}

\noindent (iv) This directly follows from (iii) and the definition of $\widehat{g_4}$.
\end{proof}
We are now ready to proceed to prove \cref{lemma:stivalues} and \cref{thm:stivalues}.
\begin{proof}[Proof of \cref{lemma:stivalues}]
Let us first show that $t_p(K)$ is monotonically decreasing.
By \cref{lem:add}~(ii), for $p \geq 2$
there exists a smooth cobordism~$C$ of genus $g_4(T_{p,p+1}) - g_4(T_{p-1,p}) = p-1$
between $T_{p-1,p}$ and $T_{p,p+1}$.
Let $F$ be a genus-minimizing slice surface of $(T_{p-1,p} \# K)^{\# n}$.
Gluing $F$ to $C^{\# n}$ gives a slice surface $F'$ of $(T_{p,p+1} \# K)^{\# n}$
of genus $g(F') = g(F) + n(p-1)$. Thus
\begin{align*}
g_4((T_{p,p+1} \# K)^{\# n}) & ~\leq~ g_4((T_{p-1,p} \# K)^{\# n}) + n(p-1) \Rightarrow \\[1ex]
\frac{g_4((T_{p,p+1} \# K)^{\# n})}{n} - \frac{p(p-1)}{2} & ~\leq~ \frac{g_4((T_{p-1,p} \# K)^{\# n})}{n} - \frac{(p-1)(p-2)}{2} \Rightarrow \\[1ex]
t_p(K) & ~\leq~ t_{p-1}(K).
\end{align*}

Next, we observe that $t_p(K)$ is bounded below, and thus converges. Indeed,
\begin{align*}
t_p(K) + t_p(-K) & =
\widehat{g_4}(T_{p,p+1} \# K) + \widehat{g_4}(T_{p,p+1} \# \,{-K}) - 2g_4(T_{p,p+1}) \\
 & \geq \widehat{g_4}(T_{p,p+1} \# K \# T_{p,p+1} \# \,{-K}) - 2g_4(T_{p,p+1}) \\
 & = 2\widehat{g_4}(T_{p,p+1}) - 2g_4(T_{p,p+1}),
\end{align*}
which is zero by \cref{lem:add}~(iv).
Hence we have $t_p(K) \geq -t_p(-K) \geq -t_1(-K)$.
Finally, taking the limit $p\to\infty$ of $t_p(K) + t_p(-K)\geq 0$ also yields $\ell(K) + \ell(-K)\geq 0$, as desired.
\end{proof}
\begin{proof}[Proof of \cref{thm:stivalues}]
We first check that $\phi(K) \in [-\ell(-K), \ell(K)]$ for every slice-torus invariant $\phi$. 
For every $p$, we have
\begin{align*}
\phi(-K) & = \phi(T_{p,p+1} \# \,{-K}) + \phi(-T_{p,p+1}) \\
&= \phi(T_{p,p+1} \# \,{-K})  - \widehat{g_4}(T_{p,p+1}) \\
     & \leq \widehat{g_4}(T_{p,p+1} \# \,{-K}) - \widehat{g_4}(T_{p,p+1}) \\
     & = t_p(-K),
\end{align*}
where we used that $\phi(J)\leq \widehat{g_4}(J)$ for all knots $J$.
Taking the limit gives $-\phi(K)=\phi(-K) \leq \ell(-K)$, and, by replacing $K$ by $-K$, we find $\phi(K) \leq \ell(K)$. Hence, we have $\phi(K) \in [-\ell(-K), \ell(K)]$ as desired.

As last step of the proof,
for a given knot $K$ and a given real number $\lambda \in [-\ell(K), \ell(-K)]$, we need to construct a slice-torus invariant $\phi$
with $\phi(K) = \lambda$.
Positive non-trivial torus knots have linearly independent Levine-Tristram signatures~\cite{MR547456}.
Therefore they are linearly independent in $\mathcal{C}\otimes\R$
and form a basis of~$\mathcal{T}$.
Thus there is a unique
homomorphism $\phi''\colon \mathcal{T} \to \mathbb{R}$ with $\phi''(T_{p,q}) = g_4(T_{p,q})$ for all coprime positive $p,q$.
We claim that
\begin{equation}\tag{$\dagger$}\label{eq:phig4}
\phi''(P) = \widehat{g_4}(P)
\end{equation}
holds for all vectors $P\in\mathcal{T}_+$.
If $P$ is a knot, then \cref{eq:phig4} is true by \cref{lem:add}~(iii).
Since $\phi''(\xi P) = \xi \phi''(P) = \xi \widehat{g_4}(P) = \widehat{g_4}(\xi P)$ for all positive rationals $\xi$,
\cref{eq:phig4} also holds for $P$ equal to a rational multiple of a knot.
Thus we have that
\[ \phi '' \vert_{\mathcal{T}_+ \cap \mathcal{C} \otimes \Q} = \widehat{g_4} \vert_{\mathcal{T}_+ \cap \mathcal{C} \otimes \Q}, \]
but $\mathcal{T}_+ \cap \mathcal{C} \otimes \Q$ is a dense subset of $\mathcal{T}_+$ endowed with the subspace topology arising from the colimit topology of the Euclidean topologies on all finite-dimensional subspaces of $\mathcal{T}$.  The colimit topology is the finest topology such that for all finite-dimensional subspaces of $\mathcal{T}$, equipped with the Euclidean topology, the inclusion homomorphism into $\mathcal{T}$ is continuous.  Moreover, $\phi''$ and the restriction of $\widehat{g_4}$ to $\mathcal{T}$ are continuous functions with respect to the colimit topology since their restrictions to all finite dimensional subspaces are continuous.  Thus \cref{eq:phig4} holds for all $P\in\mathcal{T}_+$.

Now, all $T\in \mathcal{T}$ can be written as $P - P'$ with $P, P' \in \mathcal{T}_+$.
\[
\phi''(T) = \widehat{g_4}(P) - \widehat{g_4}(P') \leq \widehat{g_4}(P -P') = \widehat{g_4}(T).
\]
So the homomorphism $\phi''$ is dominated by~$\widehat{g_4}$, i.e.~$\phi''(T) \leq \widehat{g_4}(T)$ for all $T\in\mathcal{T}$.

We now proceed to construct the desired slice-torus invariant $\phi$.
Let us first consider the case that the given knot $K$ lies in $\mathcal{T}$.
Then it follows from \cref{lem:add}~(i) that $K$ is squeezed,
and so $-\ell(-K) = \ell(K)$ by \cref{lem:sqzell}.
Therefore $\lambda = \ell(K) = \phi(K)$ for all slice-torus invariants $\phi$.
So it is enough to show the existence of any slice-torus invariant.
The Hahn-Banach theorem implies that $\phi''$ extends to a homomorphism $\phi\colon \mathcal{C}\otimes \R \to \R$
that satisfies $\phi \leq \widehat{g_4}$ on all of $\mathcal{C}\otimes\R$.
Precomposing $\phi$ with the canonical map $\mathcal{C}\to\mathcal{C}\otimes\R, K\mapsto K\otimes 1$,
gives a slice-torus invariant.

Now, let us take care of the case that $K\not\in\mathcal{T}$.
Consider the space $\mathcal{T}_K = \mathcal{T} + \langle K\rangle$.
Set $\phi'(T + \mu K) = \phi''(T) + \mu \cdot \lambda$ for all vectors $T\in\mathcal{T}$ and reals $\mu\in\mathbb{R}$.
This is clearly a homomorphism $\mathcal{T}_K \to \R$. Let us check that it is dominated by $\widehat{g_4}$, i.e.\
$\phi'(T + \mu K) \leq \widehat{g_4}(T + \mu K)$ for all $T\in\mathcal{T}$ and $\mu\in\mathbb{R}$.
We claim that the case $\mu = 1$ quickly implies the general case.
Indeed, for $\mu > 0$, assuming the case $\mu = 1$,
we have
\[
\phi'(T + \mu K) = \mu \phi'(T/\mu + K) \leq \mu \widehat{g_4}(T/\mu + K) = \widehat{g_4}(T + \mu K).
\]
The case $\mu < 0$ follows from the case $\mu > 0$
since $T + \mu K = T + (-\mu) (-K)$.
So let us now show the case $\mu = 1$, i.e.~that
for all $T\in\mathcal{T}$ we have
\begin{equation}\tag{$\ddagger$}\label{phi'g4}
\phi'(T + K) \leq \widehat{g_4}(T + K).
\end{equation}
Let us first consider the case that $T$ is a knot in $\mathcal{T}_+$.
Then by \cref{lem:add}~(i), there exists a cobordism $C$ from $T$ to some
$T_{p,p+1}$ with genus $g(C) = \widehat{g_4}(T_{p,p+1}) - \widehat{g_4}(T)$.
We then have
\begin{align*}
\phi'(T + K) & =    \widehat{g_4}(T) + \lambda \\
          & \leq \widehat{g_4}(T) + \ell(K) \\
          & \leq \widehat{g_4}(T) + t_p(K) \\
          & = \widehat{g_4}(T) + \widehat{g_4}(T_{p,p+1}\# K) - \widehat{g_4}(T_{p,p+1}) \\
          & \leq \widehat{g_4}(T) + \widehat{g_4}(T_{p,p+1} \# \,{-T}) + \widehat{g_4}(T \# K) - \widehat{g_4}(T_{p,p+1}) \\
          & \leq \widehat{g_4}(T) + g(C) + \widehat{g_4}(T \# K) - \widehat{g_4}(T_{p,p+1}) \\
          & = \widehat{g_4}(T\# K).
\end{align*}
So, we have shown \cref{phi'g4} in case that $T$ is a knot in $\mathcal{T}_+$.
If $T \in \mathcal{T}_+$ such that $n T$ is a knot for some positive integer $n$,
then
\[
\phi'(T + K) = \tfrac1{n} \phi'(n T + K^{\# n}) \leq 
\tfrac1{n} \widehat{g_4}(n T + K^{\# n}) = \widehat{g_4}(T + K).
\]
Thus \cref{phi'g4} holds for all $T$ in $\mathcal{T}_+\cap \mathcal{C} \otimes \Q$. Similarly as in the proof of \cref{eq:phig4},
the denseness of $\mathcal{T}_+\cap \mathcal{C} \otimes \Q$ in $\mathcal{T}_+$ and the continuity of $\phi'$ and $\widehat{g_4}$
now imply that \cref{phi'g4} holds for all $T\in\mathcal{T}_+$.
In the general case that $T\in\mathcal{T}$, we may again write $T$ as $P-P'$ with $P,P'\in\mathcal{T}_+$.
Applying linearity of $\phi'$ and the triangle inequality for $\widehat{g_4}$, we find
\begin{align*}
\phi'(T + K) & = \phi'(-P') + \phi'(P + K) \\
          & \leq -\widehat{g_4}(P') + \widehat{g_4}(P + K) \\
          & \leq \widehat{g_4}(T + K).
\end{align*}
This concludes the proof that $\phi'$ is dominated by $\widehat{g_4}$ on $\mathcal{T}_K$.
By the Hahn-Banach theorem, $\phi'$ extends to a homomorphism $\phi\colon \mathcal{C}\otimes\mathbb{R} \to\mathbb{R}$
that is dominated on all of its domain by $\widehat{g_4}$. Precomposing $\phi$ with $\mathcal{C} \to \mathcal{C}\otimes\mathbb{R}$
gives the desired slice-torus invariant.
\end{proof}
 
\bibliographystyle{myamsalpha}
\bibliography{References}

\providecommand{\bysame}{\leavevmode\hbox to3em{\hrulefill}\thinspace}
\providecommand{\MR}{\relax\ifhmode\unskip\space\fi MR }
\providecommand{\MRhref}[2]{%
  \href{http://www.ams.org/mathscinet-getitem?mr=#1}{#2}
}
\providecommand{\href}[2]{#2}
\begin{thebibliography}{FLL22}

\bibitem[{Baa}12]{zbMATH06092636}
S.~{Baader}: \emph{{Scissor equivalence for torus links}}, {Bull. Lond. Math.
  Soc.} \textbf{44} (2012), no.~5, 1068--1078. \xox{MR}{2975163},
  \xox{arXiv}{1011.0876}.

\bibitem[{Fel}14]{zbMATH06272609}
P.~{Feller}: \emph{{Gordian adjacency for torus knots}}, {Algebr. Geom. Topol.}
  \textbf{14} (2014), no.~2, 769--793. \xox{MR}{3159969},
  \xox{arXiv}{1301.5248}.

\bibitem[{Fel}16]{zbMATH06723030}
\bysame: \emph{{Optimal cobordisms between torus knots}}, {Commun. Anal. Geom.}
  \textbf{24} (2016), no.~5, 993--1025. \xox{MR}{3622312},
  \xox{arXiv}{1501.00483}.

\bibitem[FLL22]{sqz}
P.~Feller, L.~Lewark, and A.~Lobb: \emph{Squeezed knots}, 2022.
  \xox{arXiv}{2202.12289}.

\bibitem[KM93]{thom}
P.~B. Kronheimer and T.~S. Mrowka: \emph{Gauge theory for embedded surfaces.
  {I}}, Topology \textbf{32} (1993), no.~4, 773--826. \xox{MR}{1241873}.

\bibitem[Lew14]{lew2}
L.~Lewark: \emph{Rasmussen's spectral sequences and the
  $\mathfrak{sl}_n$-concordance invariants}, Adv. Math. \textbf{260} (2014),
  59--83. \xox{MR}{3209349}, \xox{arXiv}{1310.3100}.

\bibitem[Lit79]{MR547456}
R.~A. Litherland: \emph{Signatures of iterated torus knots}, Topology of
  low-dimensional manifolds, Lecture Notes in Math., vol. 722, Springer,
  Berlin, 1979, pp.~71--84. \xox{MR}{547456}.

\bibitem[Liv04]{livingston}
C.~Livingston: \emph{Computations of the {O}zsv\'ath-{S}zab\'o knot concordance
  invariant}, Geom. Topol. \textbf{8} (2004), 735--742. \xox{MR}{2057779},
  \xox{arXiv}{math/0311036}.

\bibitem[Liv10]{livingston_stable}
\bysame: \emph{The stable 4-genus of knots}, Algebr. Geom. Topol. \textbf{10}
  (2010), no.~4, 2191--2202. \xox{MR}{2745668}, \xox{arXiv}{0904.3054}.

\bibitem[Lob09]{lobb1}
A.~Lobb: \emph{A slice-genus lower bound from sl(n) {K}hovanov-{R}ozansky
  homology}, Adv. Math. \textbf{222} (2009), no.~4, 1220--1276.
  \xox{MR}{2554935}, \xox{arXiv}{math/0702393}.

\bibitem[Lob11]{lobb5}
\bysame: \emph{Computable bounds for {R}asmussen's concordance invariant},
  Compos. Math. \textbf{147} (2011), 661--668. \xox{MR}{2776617},
  \xox{arXiv}{0908.2745}.

\bibitem[Lob12]{lobb2}
\bysame: \emph{A note on {G}ornik's perturbation of {K}hovanov-{R}ozansky
  homology}, Algebr. Geom. Topol. \textbf{12} (2012), 293--305.
  \xox{MR}{2916277}, \xox{arXiv}{1012.2802}.

\bibitem[OS03]{osz10}
P.~Ozsv{\'a}th and Z.~Szab{\'o}: \emph{Knot {F}loer homology and the four-ball
  genus}, Geom. Topol. \textbf{7} (2003), 615--639. \xox{MR}{2026543},
  \xox{arXiv}{math/0301149}.

\bibitem[Ras03]{ras}
J.~Rasmussen: \emph{Floer homology and knot complements}, Ph.D. thesis, Harvard
  University, 2003. \xox{MR}{2704683}, \xox{arXiv}{math/0306378}.

\bibitem[Ras10]{ras3}
\bysame: \emph{Khovanov homology and the slice genus}, Invent. Math.
  \textbf{182} (2010), 419--447. \xox{MR}{2729272}, \xox{arXiv}{math/0402131}.

\bibitem[Rud93]{rudolph_QPasObstruction}
L.~Rudolph: \emph{Quasipositivity as an obstruction to sliceness}, Bull. Amer.
  Math. Soc. (N.S.) \textbf{29} (1993), no.~1, 51--59. \xox{MR}{1193540},
  \xox{arXiv}{math/9307233}.

\bibitem[Wu09]{wu3}
H.~Wu: \emph{On the quantum filtration of the {K}hovanov-{R}ozansky
  cohomology}, Adv. Math. \textbf{221} (2009), 54--139. \xox{MR}{2509322},
  \xox{arXiv}{math/0612406}.

\end{thebibliography}
\end{document}